\documentclass[10pt,leqno,letterpaper]{article}
\usepackage[utf8]{inputenc}
\usepackage{amsmath}
\usepackage{amsfonts}
\usepackage{amssymb}
\usepackage{graphicx}
\usepackage{mathrsfs}
\usepackage{upref,amsthm,amsxtra,exscale}
\usepackage{cite}
\usepackage[colorlinks=true,urlcolor=blue,
citecolor=red,linkcolor=blue,linktocpage,pdfpagelabels,
bookmarksnumbered,bookmarksopen]{hyperref}

\usepackage{subcaption}
\usepackage{caption}

\newtheorem{theorem}{Theorem}[section]
\newtheorem{corollary}[theorem]{Corollary}
\newtheorem{remark}[theorem]{Remark}
\newtheorem{lemma}[theorem]{Lemma}
\newtheorem{proposition}[theorem]{Proposition}

\numberwithin{equation}{section}

\def\r{\mathbb{R}}
\def\rn{\mathbb{R}^N}

\def\eps{\varepsilon}

\def\R{\mathbb{R}}
\def\S{\mathbb{S}}

\def\cC{\mathcal{C}}

\author{Mónica Clapp\footnote{M. Clapp was partially supported by UNAM-DGAPA-PAPIIT grant IN100718 (Mexico), CONACYT grant A1-S-10457 (Mexico).},\quad Alberto Saldaña\footnote{A. Saldaña was partially supported by the Alexander von Humboldt Foundation.},\quad and \quad Andrzej Szulkin}
\title{Phase separation, optimal partitions, and nodal solutions to the Yamabe equation on the sphere}
\date{\today}

\begin{document}
\maketitle

\begin{abstract}
We study an optimal $M$-partition problem for the Yamabe equation on the round sphere, in the presence of some particular symmetries. We show that there is a correspondence between solutions to this problem and least energy sign-changing symmetric solutions to the Yamabe equation on the sphere with precisely $M$ nodal domains.

The existence of an optimal partition is established through the study of the limit profiles of least energy solutions to a weakly coupled competitive elliptic system on the sphere.\medskip

\noindent\textsc{Keywords:} Yamabe equation on the sphere; optimal partition; $M$-nodal solution; weakly coupled competitive elliptic system.\medskip

\noindent\textsc{MSC2010: }58J05 · 58J32 · 35J50 · 35B06 · 35B08 · 35B33.
\end{abstract}

\section{Introduction and statement of results}

We study an optimal partition problem for the Yamabe equation
\begin{equation} \label{eq:1}
\mathscr{L}_gu:=-\Delta_g u + a_Nu = |u|^{2^*-2}u\qquad\text{on }\mathbb{S}^N,
\end{equation}
on the round $N$-sphere $(\mathbb{S}^N,g)$, $N\geq 3$, where $\Delta_g:=\mathrm{div}_g\nabla_g$ is the Laplace-Beltrami operator, $a_N:=\frac{N(N-2)}{4}$, and $2^*:=\frac{2N}{N-2}$ is the critical Sobolev exponent. 
More precisely, for each pair of integers $m,n\geq 2$ with $m+n=N+1$ and every $[O(m)\times O(n)]$-invariant open subset $U$ of $\mathbb{S}^N$, we consider the problem
\begin{equation} \label{eq:3}
\begin{cases}
\mathscr{L}_gu = |u|^{2^*-2}u &\text{ in }U,\\
u=0 &\text{ on }\partial U,\\
u\text{ is }[O(m)\times O(n)]\text{-invariant},
\end{cases}
\end{equation}
and denote by $c^{(m,n)}_U$ the least energy of a nontrivial solution to \eqref{eq:3}, i.e.,
$$c^{(m,n)}_U:=\inf\left\{\frac{1}{N}\int_U|u|^{2^*}:u\neq0,\;u\text{ solves }\eqref{eq:3}\right\}.$$
Given $M\geq 2$, we look for a solution to the optimal $M$-partition problem
\begin{equation} \label{eq:4}
\inf_{\{U_1,\ldots,U_M\}\in\mathscr{P}_M^{(m,n)}}\;\sum_{i=1}^M c^{(m,n)}_{U_i},
\end{equation}
on $\mathbb{S}^N$, where
\begin{align*}
\mathscr{P}_M^{(m,n)}:=\{\{U_1,\ldots,U_M\}:\;& U_i\neq\emptyset \text{ is }[O(m)\times O(n)]\text{-invariant and open in }\mathbb{S}^N\\
&\text{ and }U_i\cap U_j=\emptyset\text{ if }i\neq j,\quad \forall i,j=1,\ldots,M \}.
\end{align*}
A solution to \eqref{eq:4} is a set $\{U_1,\ldots,U_M\}\in\mathscr{P}_M^{(m,n)}$ such that
$$\sum_{i=1}^M c^{(m,n)}_{U_i}=\inf_{\{V_1,\ldots,V_M\}\in\mathscr{P}_M^{(m,n)}}\;\sum_{i=1}^M c^{(m,n)}_{V_i}.$$
An outstanding feature of this optimal partition problem is that any solution of \eqref{eq:4} is the set of nodal domains of an $[O(m)\times O(n)]$-invariant sign-changing solution to the Yamabe problem \eqref{eq:1}, which has minimal energy among all $[O(m)\times O(n)]$-invariant $M$-nodal solutions to \eqref{eq:1}. This fact is proved in Theorem \ref{thm:partition} below.

In order to establish the existence of a solution to the problem \eqref{eq:4}, we consider the competitive elliptic system
\begin{equation} \label{eq:2}
\begin{cases}
\mathscr{L}_g u_i = |u_i|^{2^*-2}u_i + \sum\limits_{j\neq i} \lambda_{ij}\beta_{ij}|u_j|^{\alpha_{ij}}|u_i|^{\beta_{ij}-2}u_i\qquad\text{on }\mathbb{S}^N, \\
u_i\text{ is }[O(m)\times O(n)]\text{-invariant},\qquad i,j=1,\ldots,M,
\end{cases}
\end{equation}
where $\lambda_{ij}=\lambda_{ji}<0$, $\alpha_{ij},\beta_{ij}>1$, $\alpha_{ij}=\beta_{ji}$, and $\alpha_{ij}+\beta_{ij}=2^*$.

The existence of a positive least energy fully nontrivial solution to this system was recently shown in \cite{cs}. Fully nontrivial means that every component $u_i$ is nontrivial. Here, we show that this system exhibits phase separation as $\lambda_{ij}\to-\infty$ and that this phenomenon gives rise to a solution to \eqref{eq:4} and to an $M$-nodal solution of the Yamabe problem \eqref{eq:1}. The precise statement is given by the following theorem.

We write $\mathbb{B}^d$ and $\mathbb{S}^{d-1}$ for the open unit ball and the unit sphere in $\mathbb{R}^d$, respectively. The symbol ``$\cong$" stands for ``is $[O(m)\times O(n)]$-diffeomorphic to".

\begin{theorem} \label{thm:main}
Let $m,n\geq 2$ with $m+n=N+1$ and, for each $i,j=1,\ldots,M$, $i\ne j$, let $(\lambda_{ij,k})$ be a sequence of negative numbers such that $\lambda_{ij,k}\to -\infty$ as $k\to\infty$. Let $u_{k}=(u_{k,1},\ldots,u_{k,M})$ be a positive least energy fully nontrivial $[O(m)\times O(n)]$-invariant solution to the system \eqref{eq:2} with $\lambda_{ij}=\lambda_{ij,k}$. Then, after passing to a subsequence, we have that
\begin{itemize}
\item[$(a)$]$u_{k,i}\to u_{\infty,i}$ strongly in $H^1_g(\mathbb{S}^N)$, $u_{\infty,i}\geq 0$, $u_{\infty,i}$ is continuous on $\mathbb{S}^N$ and $u_{\infty,i}|_{U_i}$ is a least energy solution to problem \eqref{eq:3} in $U_i:=\{x\in\mathbb{S}^{N}:u_{\infty,i}(x)>0\}$, for each $i=1,\ldots,M$.
\item[$(b)$]$\{U_1,\ldots,U_M\}\in \mathscr{P}_M^{(m,n)}$ and it solves the optimal $M$-partition problem \eqref{eq:4} on $\mathbb{S}^N$.
\item[$(c)$]$U_1,\ldots,U_M$ are smooth and connected, $\overline{U_1\cup\cdots\cup U_M}=\mathbb{S}^{N}$ and, after relabeling, we have that
\begin{itemize}
\item[•] $U_1\cong\mathbb{S}^{m-1}\times \mathbb{B}^{n}$,\quad $U_i\cong\mathbb{S}^{m-1}\times\mathbb{S}^{n-1}\times(0,1)$ if  $i=2,\ldots,M-1$, and\quad $U_M\cong\mathbb{B}^m\times \mathbb{S}^{n-1}$,
\item[•] $\overline{U}_i\cap \overline{U}_{i+1}\cong\mathbb{S}^{m-1}\times\mathbb{S}^{n-1}$ and\quad $\overline{U}_i\cap \overline{U}_j=\emptyset$\, if\, $|j-i|\geq 2$,
\item[•] the function
$$u:=\sum_{i=1}^M(-1)^{i-1}u_{\infty,i}$$
is an $[O(m)\times O(n)]$-invariant sign-changing solution to the Yamabe problem \eqref{eq:1} with precisely $M$ nodal domains and $u$ has least energy among all such solutions.
\end{itemize}
\end{itemize}
\end{theorem}

As we mentioned before, the existence of a positive least energy fully nontrivial solution to the system \eqref{eq:2} was established in \cite{cs}. So Theorem \ref{thm:main} yields the following result.

\begin{corollary} \label{cor:main}
For any pair of integers $m,n\geq 2$ with $m+n=N+1$ and any $M\geq 2$, the following statements hold true:
\begin{itemize}
\item[$(i)$]There exists a solution to the optimal $M$-partition problem \eqref{eq:4} on the round sphere $\mathbb{S}^N$.
\item[$(ii)$]There exists a least energy $[O(m)\times O(n)]$-invariant sign-changing solution to the Yamabe problem \eqref{eq:1} with precisely $M$ nodal domains.
\end{itemize}
\end{corollary}

For each pair of integers $m,n\geq 2$ with $m+n=N+1$, W.Y. Ding established the existence of infinitely many $[O(m)\times O(n)]$-invariant sign-changing solutions to the problem \eqref{eq:1} in \cite{d}. 

A significant feature of these symmetries is that the space of $[O(m)\times O(n)]$-orbits in $\mathbb{S}^N$ is one-dimensional; see \eqref{eq:q}. This allows us to derive the continuity of the limit profiles $u_{\infty,i}$ of the least energy solutions to the system \eqref{eq:2} and to obtain a solution to the optimal $M$-partition problem \eqref{eq:4}; see Proposition~\ref{prop:continuity} and Theorem \ref{thm:phase_separation}. It also allows us to show that (after adding the two exceptional orbits $\S^{m-1}\times\{0\}$ and $\{0\}\times \S^{n-1}$) any solution to the optimal $M$-partition problem \eqref{eq:4} has the properties stated in $(c)$ of Theorem \ref{thm:main}. In particular, it is the set of nodal domains of a least energy $[O(m)\times O(n)]$-invariant $M$-nodal solution to the Yamabe problem \eqref{eq:1}; see Theorem \ref{thm:partition}. 

We prove in addition that, conversely, the set of nodal domains of a least energy $[O(m)\times O(n)]$-invariant $M$-nodal solution to \eqref{eq:1} solves \eqref{eq:4}; see Corollary \ref{cor:nodal}. This characterizes the close relationship between solutions to \eqref{eq:4} and least energy $[O(m)\times O(n)]$-invariant $M$-nodal solutions to \eqref{eq:1} on $\mathbb{S}^N$.

In \cite{fp} Fernández and Petean use the one-dimensionality of the orbit space to reduce problem \eqref{eq:1} to an ODE and they show the existence of an $[O(m)\times O(n)]$-invariant solution with precisely $M$ nodal domains via a double-shooting method. Our approach is independent of ODE techniques and it readily guarantees that the obtained $M$-nodal solution has least energy among all $[O(m)\times O(n)]$-invariant sign-changing solutions to the Yamabe problem \eqref{eq:1} with at least $M$ nodal domains.  We remark that it is not obvious to determine if the solutions given by Theorem~\ref{thm:main} and those obtained in \cite{fp} are the same or not. 

For a subcritical competitive elliptic system of two equations, the relation between phase separation, optimal $2$-partitions and $2$-nodal solutions to an elliptic equation was first established by Conti, Terracini and Verzini in \cite{ctv}. Theorem~\ref{thm:main} for $M=2$ was proved in \cite{cp}. The case $M=2$ is relatively simple because, as shown in \cite{ccn}, a $2$-nodal solution for the equation can be obtained by minimization of the energy functional on a suitable constraint. So one needs only to show that the sum of the limit profiles of the two components of the system, with opposite signs, is a minimizer. This immediately yields the continuity properties required to get an optimal partition; see \cite{cp}.

For $M>2$ the problem is, in general, much harder because there is no suitable constraint which gives rise to sign-changing solutions with precisely $M$ nodal domains via minimization. The relation between phase separation and optimal $M$-partitions has been studied, e.g., in \cite{tt,tt2,rtt} and some of the references therein. One main difficulty consists in establishing the uniform Hölder continuity of the solutions to the system \eqref{eq:2}, which is needed to derive some regularity of the limit profiles. This delicate question has been handled in \cite{nttv,tt2}. Another sensitive issue would be to determine whether these limit profiles can be ordered in such a way that their sum, with alternating signs, is a sign-changing solution to a related equation. This is not true in general.

In the situation considered in this paper, the symmetries are of help to treat both of these questions and to obtain the precise description of the topological nature of the optimal partition described in statement $(c)$ of Theorem \ref{thm:main}. 

It is worth adding that sign-changing solutions to the Yamabe problem \eqref{eq:1} on the round sphere, which are not $[O(m)\times O(n)]$-invariant, have been obtained in \cite{c,dmpp,fp}. 

This paper is organized as follows: Section \ref{sec:preliminaries} contains some preliminary material. In Section \ref{sec:euclidean} we translate the problems on a sphere to problems in a Euclidean space and, in Section \ref{sec:proof}, we prove our main results.

\section{Preliminaries} \label{sec:preliminaries}

Let $(\mathbb{S}^{N},g)$ be the round sphere and $p\in\mathbb{S}^{N}$ its north pole. The stereographic projection $\sigma:\mathbb{S}^{N}\smallsetminus\{p\}\to\mathbb{R}^{N}$ is a
conformal diffeomorphism. The coordinates of the standard metric $g$ in the chart given by $\sigma^{-1}:\mathbb{R}^{N}\to\mathbb{S}^{N}\smallsetminus\{p\}$ are $g_{ij}=\psi^{2^*-2}\delta_{ij}$, where
$$\psi(x):=\left(\frac{2}{1+|x|^{2}}\right)^{(N-2)/2},\qquad x\in\mathbb{R}^{N}.$$
Recall that $a_N := \frac{N(N-2)}4$. For $u\in\mathcal{C}^{\infty}(\mathbb{S}^{N})$, we set $v(x):=\psi(x)\,u(\sigma^{-1}(x))$. Then,
\begin{equation} \label{eq:laplacian}
\mathscr{L}_gu\circ\sigma^{-1}=\left(-\Delta_g u + a_Nu\right)\circ\sigma^{-1} = -\psi^{1-2^*}\Delta v\quad\text{in }\mathbb{R}^N;
\end{equation}
see, e.g., \cite[Proposition 6.1.1]{h}. This yields an equivalence between the Yamabe problem \eqref{eq:1} on $(\mathbb{S}^{N},g)$ and the problem
\begin{equation} \label{eq:1a}
-\Delta v = |v|^{2^*-2}v,\qquad v\in D^{1,2}(\mathbb{R}^N),
\end{equation}
where, as usual, $D^{1,2}(\mathbb{R}^N):=\{v\in L^{2^*}(\mathbb{R}^N):\nabla v\in L^2(\mathbb{R}^N,\mathbb{R}^N)\}$.

Fix $m,n\geq 2$ with $m+n=N+1$ and set $\Gamma:=O(m)\times O(n)$. A function $u:\mathbb{S}^N\to\mathbb{R}$ is $\Gamma$-\emph{invariant} if 
$$u(\gamma z)=u(z)\qquad \text{for every }\,\gamma\in\Gamma,\;z\in\mathbb{S}^{N}.$$ 
For each $\gamma\in\Gamma$, consider the map $\widetilde{\gamma}:=\sigma\circ\gamma^{-1}\circ\sigma^{-1}:\mathbb{R}^{N}\to\mathbb{R}^{N}$, which is well defined except at a single point. This gives a conformal action of $\Gamma$ on $\mathbb{R}^N$. We say that a function $v:\mathbb{R}^N\to\mathbb{R}$ is $\Gamma$-\emph{invariant} if
$$|\det\widetilde{\gamma}'(x)\,|^{1/2^*}v(\widetilde{\gamma}x)=v(x)\qquad \text{for every }\,\gamma\in\Gamma,\;x\in\mathbb{R}^{N}.$$
Noting that
$$|\det\widetilde{\gamma}'(x)| = \left(\frac{\psi(x)}{\psi(\widetilde{\gamma}(x))}\right)^{2^*},$$
we conclude that $u:\mathbb{S}^N\to\mathbb{R}$ is $\Gamma$-invariant iff $v:=\psi(u\circ\sigma^{-1}):\mathbb{R}^N\to\mathbb{R}$ is $\Gamma$-invariant. See \cite[Section 3]{cp} for more details.

As usual, let $H_g^1(\mathbb{S}^{N})$ be the closure of $\mathcal{C}^{\infty}(\mathbb{S}^{N})$ with respect to the norm $\|u\|_g:=\left(\int_{\mathbb{S}^{N}}(|\nabla_{g}u|_{g}^{2}+a_Nu^2)\mathrm{d}V_{g}\right)^{1/2}$, and let $H^1_g(\mathbb{S}^{N})^\Gamma$ and $D^{1,2}(\mathbb{R}^{N})^\Gamma$ denote the spaces of $\Gamma$-invariant functions in $H^1_g(\mathbb{S}^{N})$ and $D^{1,2}(\mathbb{R}^{N})$ respectively. 

\begin{lemma} \label{lem:isometry1}
If $u\in\mathcal{C}^{\infty}(\mathbb{S}^{N})$ and $v:=\psi(u\circ\sigma^{-1})$, then
$$\|u\|^2_g=\int_{\mathbb{S}^{N}}(|\nabla_{g}u|_{g}^{2}+a_Nu^2)\mathrm{d}V_{g} = \int_{\mathbb{R}^{N}}|\nabla v|^2\mathrm{d}x.$$
Therefore, the mapping $\mathscr{I}:H^1_g(\mathbb{S}^{N})^\Gamma\to D^{1,2}(\mathbb{R}^{N})^\Gamma$, given by $\mathscr{I} u:=\psi(u\circ\sigma^{-1})$, is an isometric isomorphism.
\end{lemma}

\begin{proof}
The volume element on $(\mathbb{S}^{N},g)$ is $\mathrm{d}V_{g}=\sqrt{\det(g_{ij})}\,\mathrm{d}x=\psi^{2^*}\mathrm{d}x$. So, multiplying \eqref{eq:laplacian} by $u\circ\sigma^{-1}$ and integrating by parts, yields the identity; see \cite[Section 3]{cp} for more details.
\end{proof}

A crucial property of the $\Gamma$-action is the following one.

\begin{lemma} \label{lem:compactness}
The embeddings $H^1_g(\mathbb{S}^N)^\Gamma\hookrightarrow L_g^{2^*}(\mathbb{S}^N)$,\; $D^{1,2}(\mathbb{R}^N)^\Gamma\hookrightarrow L^{2^*}(\mathbb{R}^N)$ are compact.
\end{lemma}

\begin{proof}
Since the dimension of every $\Gamma$-orbit in $\mathbb{S}^N$ is at least $\min\{m-1,n-1\}\geq 1$, by \cite[Corollary 1]{hv} we have that $H^1_g(\mathbb{S}^N)^\Gamma\hookrightarrow L_g^{2^*}(\mathbb{S}^N)$ is compact. The statement for $\mathbb{R}^N$ follows from Lemma \ref{lem:isometry1}.
\end{proof}

The $\Gamma$-orbit space of $\S^N$, i.e., the quotient space obtained by identifying each $\Gamma$-orbit $\Gamma z:=\{\gamma z:\gamma\in\Gamma\}$ in $\S^N$ to a single point, may be described as follows. We write the points in $\S^N$ as $z=(z_1,z_2)$ with $z_1\in\R^m,\ z_2\in\R^n$, and define $q:\S^N\to[0,\pi]$ by
\begin{equation} \label{eq:q}
q(z_1,z_2)=\arccos(|z_1|^2-|z_2|^2).
\end{equation}
This function is a quotient map which identifies each $\Gamma$-orbit in $\S^N$ to a single point. So the $\Gamma$-orbit space of $\S^N$ is one-dimensional. Note that
$$q^{-1}(0)\cong\mathbb{S}^{m-1},\qquad q^{-1}(t)\cong\mathbb{S}^{m-1}\times\mathbb{S}^{n-1}\text{ if  }t\in(0,\pi),\qquad q^{-1}(\pi)\cong\mathbb{S}^{n-1}.$$
We call $q$ the $\Gamma$-\emph{orbit map} of $\S^N$.

Next, we describe the norm induced by $\|\cdot\|_g$ in $\mathcal{C}^\infty[0,\pi]$, via the $\Gamma$-orbit map. Our intention is to take advantage of the one-dimensionality of the $\Gamma$-orbit space to deduce some continuity properties of the functions in $H^1_g(\mathbb{S}^{N})^\Gamma$; see Proposition \ref{prop:continuity}.

Let $H_h^1(0,\pi)$ be the closure of $\mathcal{C}^\infty[0,\pi]$ with respect to the norm
\begin{align*}
 \|w\|_h:=\left(\int_0^\pi \left(|w'(t)|^2 + \frac{a_N}{4}|w|^{2} \right)h(t)\ dt\right)^{\frac{1}{2}},
\end{align*}
where 
$$h(t):=2\, |\S^{m-1}|\, |\S^{n-1}|\,\cos^{m-1}(\frac{t}{2})\,\sin^{n-1}(\frac{t}{2}).$$

\begin{lemma} \label{lem:isometry2} 
For every $u\in\mathcal{C}^\infty(\S^N)^\Gamma$ there exists a unique $w\in\mathcal{C}^\infty[0,\pi]$ such that $u=w\circ q$ and
$$\|u\|_g^2=\int_{\mathbb{S}^{N}}(|\nabla_{g}u|_{g}^{2}+a_Nu^2)\mathrm{d}V_{g} = \|w\|_h^2.$$
Therefore, the mapping $\mathscr{J}:H^1_h(0,\pi)\to H^1_g(\mathbb{S}^{N})^\Gamma$, given by $\mathscr{J}w:=w\circ q$, is an isometric isomorphism.
\end{lemma}

\begin{proof}
Let $f:\S^N\to[-1,1]$ denote the function 
$$f(z_1,z_2)=|z_1|^2-|z_2|^2.$$ 
Then, $\nabla_g f(z_1,z_2)=4(|z_2|^2z_1 ,-|z_1|^2z_2)$ and
$$|\nabla_g f(z_1,z_2)|^2 = 16 |z_1|^2|z_2|^2 = (b\circ f)(z_1,z_2),$$
where $b:[-1,1]\to\R$ is given by $b(t)=4(1-t^2)$. 

Clearly, for every $u\in\mathcal{C}^\infty(\S^N)^\Gamma$, there exists a unique $w\in\mathcal{C}^\infty[0,\pi]$ such that 
$$u=w\circ q=\phi\circ f,\quad \text{with } \phi=w\circ \arccos.$$
 As $\nabla_g u = (\phi'\circ f)\nabla_g f$, we get that
\begin{align*}
|\nabla_g u|^2 = |\phi'\circ f|^2 (b\circ f) =(|\phi'|^2 b)\circ f = \theta\circ f,\quad \text{with }\theta:=|\phi'|^2 b.
\end{align*}
A straightforward computation (see \cite[Lemma 2.2]{fp}) gives
\begin{align} \label{eq:equality}
\int_{\S^N}|\nabla_g u|^2 \ \mathrm{d}V_g=\int_{\S^N} \theta\circ f\ \mathrm{d}V_g
 =\frac{1}{4}\int_0^\pi \theta(\cos(t))h(t)\ \mathrm{d}t.
\end{align}
Since $\phi'(s)=w'(\arccos (s))\left( \frac{-1}{\sqrt{1-s^2}} \right)$, setting $s=\cos t$ we get that
\begin{align*}
 \theta(\cos(t))=|\phi'(\cos(t))|^2\, b(\cos(t)) =|w'(t)|^2\frac{1}{\sin^2t}\,4(1-\cos^2t)=4|w'(t)|^2.
\end{align*}
Hence,
$$\int_{\S^N}|\nabla_g u|^2 \ \mathrm{d}V_g =\int_0^\pi |w'(t)|^2h(t)\ \mathrm{d}t.$$
Similarly, taking \,$\theta := w^2\circ\arccos$\, in the second identity in \eqref{eq:equality}, one sees that
$$\int_{\S^N}u^2 \ \mathrm{d}V_g =\frac{1}{4}\int_0^\pi |w(t)|^2h(t)\ \mathrm{d}t.$$
This completes the proof.
\end{proof}

The following fact plays an important role in the proof of our main result; see Theorem \ref{thm:phase_separation}.

\begin{proposition} \label{prop:continuity}
Let $Z:=(\S^{m-1}\times\{0\})\, \cup\, (\{0\}\times\S^{n-1})\subset\S^N$. For every $u\in H^1_g(\S^N)^\Gamma$ there exists $\bar u\in \cC^0(\S^N\smallsetminus Z)$ such that $u=\bar u$ a.e. in $\S^N$.
\end{proposition}

\begin{proof}
For every $\eps\in(0,\frac{\pi}{2})$, the norm $\|\cdot\|_h$ in $H_h^1(\eps,\pi-\eps)$ is equivalent to the standard norm in $H^1(\eps,\pi-\eps)$. Hence, $H_h^1(\eps,\pi-\eps)=H^1(\eps,\pi-\eps)\subset\cC^0(\eps,\pi-\eps)$ for every $\eps\in(0,\frac{\pi}{2})$. The claim now follows from Lemma \ref{lem:isometry2}.
\end{proof}
  
\begin{remark}
Observe that there are functions in $H_h^1(0,\pi)$ which are singular at $0$ and at $\pi$; for example, $w(t) = \ln(-\ln(\frac{t}{2\pi})) + \ln(-\ln(\frac{\pi - t}{2\pi}))$ belongs to $H_h^1(0,\pi)$.
\end{remark}

\section{The result in Euclidean space} \label{sec:euclidean}

As before, we fix $m,n\geq 2$ with $m+n=N+1$ and write $\Gamma:=O(m)\times O(n)$. We consider the conformal action of $\Gamma$ on $\mathbb{R}^N$ introduced in Section \ref{sec:preliminaries}. So, a subset $X$ of $\mathbb{R}^N$ is \emph{$\Gamma$-invariant} if 
$$\widetilde{\gamma}x:=(\sigma\circ\gamma^{-1}\circ\sigma^{-1})(x)\in X\qquad\forall\gamma\in\Gamma,\;\forall x\in X.$$

Using the identity \eqref{eq:laplacian} it is readily seen that the competitive system \eqref{eq:2} on $\S^N$ is equivalent to the competitive elliptic system in $\mathbb{R}^N$
\begin{equation} \label{eq:2a}
\begin{cases}
-\Delta v_i = |v_i|^{2^*-2}v_i + \sum\limits_{j\neq i} \lambda_{ij}\beta_{ij}|v_j|^{\alpha_{ij}}|v_i|^{\beta_{ij}-2}v_i, \\
v_i\in D^{1,2}(\mathbb{R}^N)^\Gamma,\qquad i,j=1,\ldots,M.
\end{cases}
\end{equation}
More precisely, setting $v_i(x):=\psi(x)u_i(\sigma^{-1}(x))$, we have that $(u_1,\ldots,u_M)$ solves \eqref{eq:2} iff $(v_1,\ldots,v_M)$ solves \eqref{eq:2a}.

We write $\|\cdot\|$ and $|\cdot|_{2^*}$ for the norms in $D^{1,2}(\mathbb{R}^N)$ and $L^{2^*}(\mathbb{R}^N)$, i.e.,
$$\|v\|^2:=\int_{\mathbb{R}^N}|\nabla v|^2,\qquad|v|^{2^*}_{2^*}:=\int_{\mathbb{R}^N}|v|^{2^*},$$
and consider the Hilbert space $\mathcal{H}:=(D^{1,2}(\mathbb{R}^N)^\Gamma)^M$ with the obvious norm. The functional $\mathcal{J}:\mathcal{H}\to\mathbb{R}$ given by 
$$\mathcal{J}(v_1,\ldots,v_M) := \frac{1}{2}\sum_{i=1}^M\|v_i\|^2 - \frac{1}{2^*}\sum_{i=1}^M|v_i|^{2^*}_{2^*} - \frac{1}{2}\sum_{j\neq i}\int_{\mathbb{R}^N}\lambda_{ij}|v_j|^{\alpha_{ij}}|v_i|^{\beta_{ij}},$$
is of class $\mathcal{C}^1$ and, since $\lambda_{ij}=\lambda_{ji}$ and $\beta_{ij}=\alpha_{ji}$, we have that
\begin{align*}
\partial_i\mathcal{J}(v_1,\ldots,v_M)[v]=\int_{\mathbb{R}^N}\nabla v_i\cdot\nabla v &- \int_{\mathbb{R}^N}|v_i|^{2^*-2}v_iv \\
&- \sum_{j\neq i}\int_{\mathbb{R}^N}\lambda_{ij}\beta_{ij}|v_j|^{\alpha_{ij}}|v_i|^{\beta_{ij}-2}v_iv,
\end{align*}
for any $v\in D^{1,2}(\mathbb{R}^N)^\Gamma$, $i=1,\ldots,M$. So the critical points of $\mathcal{J}$ are the solutions to the system \eqref{eq:2a}; see \cite{cs}. The fully nontrivial ones belong to the set
\begin{equation*}
\mathcal{N}^\Gamma := \{(v_1,\ldots,v_M)\in\mathcal{H}:v_i\neq 0, \;\partial_i\mathcal{J}(v_1,\ldots,v_M)[v_i]=0, \; \forall i=1,\ldots,M\}.
\end{equation*}
Note that 
\begin{equation*}
\mathcal{J}(v_1,\ldots,v_M) = \frac{1}{N}\sum_{i=1}^M\|v_i\|^2\qquad\text{if }(v_1,\ldots,v_M)\in\mathcal{N}^\Gamma.
\end{equation*}
It is shown in \cite[Theorem 1.2]{cs} that $\inf_{\mathcal{N}^\Gamma}J$ is attained at some $(v_1,\ldots,v_M)\in\mathcal{N}^\Gamma$ with $v_i\geq 0$.

On the other hand, the optimal $M$-partition problem \eqref{eq:4} on $\mathbb{S}^N$ is equivalent to an optimal $M$-partition problem in $\mathbb{R}^N$. 
Namely, if $\Omega$ is a $\Gamma$-invariant open subset of $\mathbb{R}^N$, we denote by $D^{1,2}_0(\Omega)^\Gamma$ the space of $\Gamma$-invariant functions in $D^{1,2}_0(\Omega)$, where as usual $D^{1,2}_0(\Omega)$ is the closure of $\mathcal{C}_c^\infty(\Omega)$ in $D^{1,2}(\mathbb{R}^N)$, and we consider the energy functional and the Nehari manifold
\begin{align*}
&J_\Omega(v):=\frac{1}{2}\int_\Omega|\nabla v|^2-\frac{1}{2^*}\int_\Omega|v|^{2^*},\\
&\mathcal{M}_\Omega^\Gamma:=\{v\in D^{1,2}_0(\Omega)^\Gamma:v\neq 0,\;J_\Omega'(v)v=0\},
\end{align*}
associated to the problem
\begin{equation} \label{eq:3a}
-\Delta v = |v|^{2^*-2}v, \qquad v\in D^{1,2}_0(\Omega)^\Gamma.
\end{equation}
Then, \eqref{eq:4} is equivalent to the optimal $M$-partition problem
\begin{equation} \label{eq:4a}
\inf_{\{\Omega_1,\ldots,\Omega_M\}\in\mathcal{P}_M^\Gamma}\;\sum_{i=1}^M c_{\Omega_i}^\Gamma,\qquad\text{where }c_{\Omega_i}^\Gamma:=\inf_{\mathcal{M}^\Gamma_{\Omega_i}}J_{\Omega_i}
\end{equation}
and
\begin{align*}
\mathcal{P}_M^\Gamma:=\{\{\Omega_1,\ldots,\Omega_M\}:\;& \Omega_i\neq\emptyset \text{ is }\Gamma\text{-invariant and open in }\mathbb{R}^N\;\forall i=1,\ldots,M, \\
&\text{ and }\Omega_i\cap\Omega_j=\emptyset\text{ if }i\neq j\}.
\end{align*}
More precisely, setting $U_i:=\sigma^{-1}(\Omega_i)$ where $\sigma$ is the stereographic projection, we have that $\{U_1,\ldots,U_M\}$ solves the optimal $M$-partition problem \eqref{eq:4} on $\mathbb{S}^N$ iff $\{\Omega_1,\ldots,\Omega_M\}$ solves the optimal $M$-partition problem \eqref{eq:4a} in $\mathbb{R}^N$.

Note that, if $\{\Omega_1,\ldots,\Omega_M\}\in\mathcal{P}_M^\Gamma$ and $v_i\in\mathcal{M}_{\Omega_i}^\Gamma$ then, since $v_iv_j=0$ for $i\neq j$, we have that $(v_1,\ldots,v_M)\in\mathcal{N}^\Gamma$ and $\mathcal{J}(v_1,\ldots,v_M)=J_{\Omega_1}(v_1)+\cdots+J_{\Omega_M}(v_M)$. Therefore, $\inf_{\mathcal{N}^\Gamma}\mathcal{J}\leq c_{\Omega_1}^\Gamma+\cdots+c_{\Omega_M}^\Gamma$ and, consequently,
\begin{equation} \label{eq:comparison}
\inf_{\mathcal{N}^\Gamma}\mathcal{J}\leq \inf_{\{\Omega_1,\ldots,\Omega_M\}\in\mathcal{P}_M^\Gamma}\;\sum_{i=1}^M c_{\Omega_i}^\Gamma.
\end{equation}

Theorem \ref{thm:main} can be restated as follows.

\begin{theorem} \label{thm:main_RN}
For each $i,j=1,\ldots,M$, $i\ne j$, let $(\lambda_{ij,k})$ be a sequence of negative numbers such that $\lambda_{ij,k}\to -\infty$ as $k\to\infty$, and  let $v_{k}=(v_{k,1},\ldots,v_{k,M})$ be a positive least energy fully nontrivial $\Gamma$-invariant solution to the system \eqref{eq:2a} with $\lambda_{ij}=\lambda_{ij,k}$. Then, after passing to a subsequence, we have that
\begin{itemize}
\item[$(a)$]$v_{k,i}\to v_{\infty,i}$ strongly in $D^{1,2}(\mathbb{R}^N)$,\, $v_{\infty,i}\geq 0$,\, $v_{\infty,i}$ is continuous and $v_{\infty,i}|_{\Omega_i}$ is a least energy solution to the problem \eqref{eq:3a} in $\Omega_i:=\{x\in\mathbb{R}^{N}:v_{\infty,i}(x)>0\}$, for each $i=1,\ldots,M$.
\item[$(b)$]$\{\Omega_1,\ldots,\Omega_M\}\in\mathcal{P}^\Gamma_M$ and it solves the optimal $M$-partition problem \eqref{eq:4a} in $\mathbb{R}^{N}$.
\item[$(c)$]$\Omega_1,\ldots,\Omega_M$ are smooth and connected, $\overline{\Omega_1\cup\cdots\cup \Omega_M}=\mathbb{R}^{N}$ and, after reordering, we have that $\Omega_1,\ldots,\Omega_{M-1}$ are bounded, $\Omega_M$ is unbounded,
\begin{itemize}
\item[$(c_1)$] $\Omega_1\cong\mathbb{S}^{m-1}\times \mathbb{B}^{n}$,\quad $\Omega_i\cong\mathbb{S}^{m-1}\times\mathbb{S}^{n-1}\times(0,1)$ if  $i=2,\ldots,M-1$, and\quad $\Omega_M\cup\{\infty\}\cong\mathbb{B}^m\times \mathbb{S}^{n-1}$,
\item[$(c_2)$] $\overline{\Omega}_i\cap \overline{\Omega}_{i+1}\cong\mathbb{S}^{m-1}\times\mathbb{S}^{n-1}$ and\quad $\overline{\Omega}_i\cap \overline{\Omega}_j=\emptyset$\, if\, $|j-i|\geq 2$,
\item[$(c_3)$] the function
$$v:=\sum_{i=1}^M(-1)^{i-1}v_{\infty,i}$$
is a $\Gamma$-invariant sign-changing solution to the problem \eqref{eq:1a} with precisely $M$ nodal domains and $v$ has least energy among all such solutions.
\end{itemize}
\end{itemize}
\end{theorem}

We prove this result in the following section.

\begin{figure}[ht]
\begin{center}
\includegraphics[width=.35\textwidth]{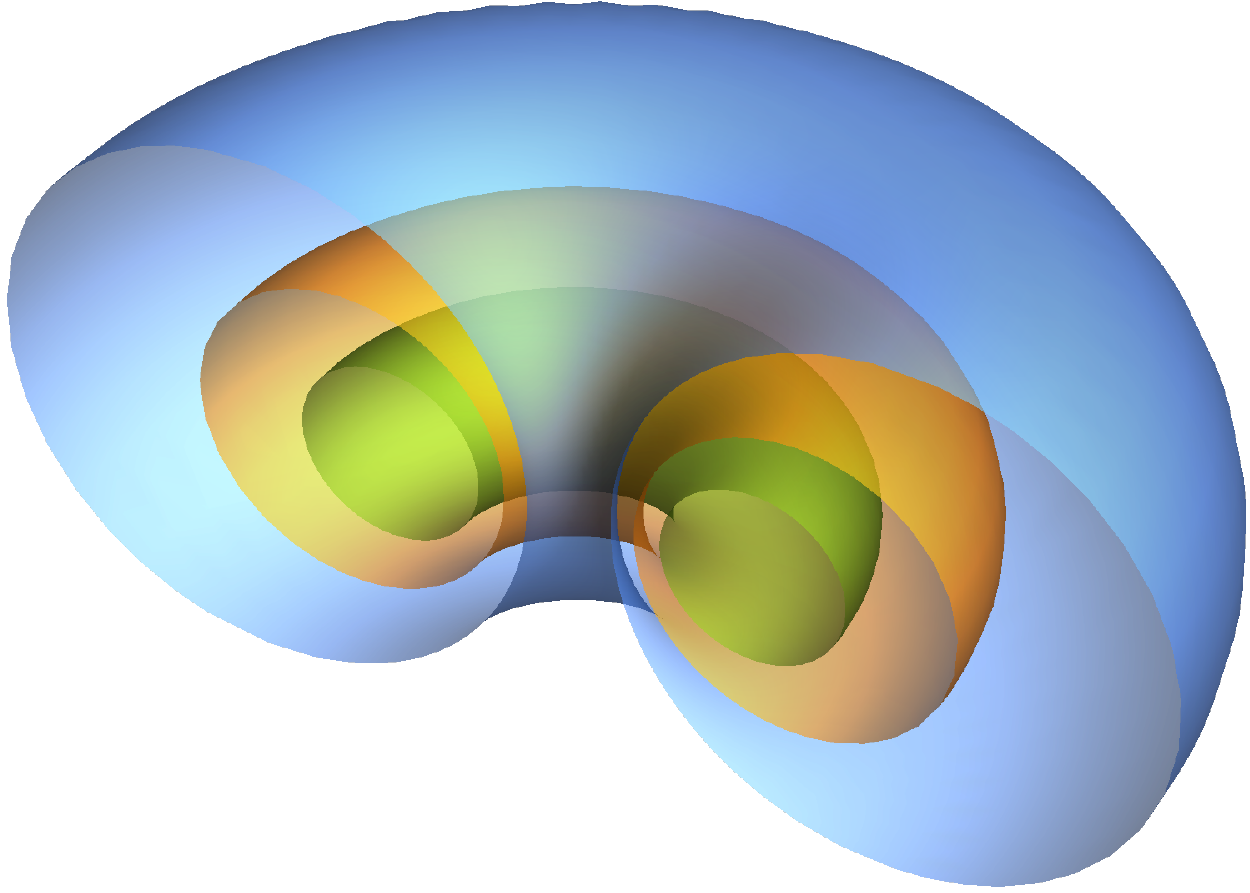}
\caption{\small{Transversal cut illustrating the optimal partition $\{\Omega_1,\ldots,\Omega_4\}$ of $\mathbb{R}^3$ given by Theorem \ref{thm:main_RN}. $\Omega_1$ is the interior of the innermost torus, $\Omega_2$ and $\Omega_3$ are the domains between two consecutive tori, and $\Omega_4$ is the exterior of the outermost torus.}}
\end{center}
\end{figure}

\section{The proof of the main result} \label{sec:proof}

Theorem \ref{thm:main_RN} follows from the next two theorems, which are of independent interest. Let
$$\widetilde{q}:=q\circ\sigma^{-1}:\mathbb{R}^N\to[0,\pi],$$
where $\sigma$ is the stereographic projection and $q$ is the $\Gamma$-orbit map of $\S^N$ defined in \eqref{eq:q}. Writing $\rn=\r^m\times\r^{n-1}$, it is easy to see that $\widetilde q^{\,-1}(0) = \mathbb{S}^{m-1}\times\{0\}$ and $\widetilde q^{\,-1}(\pi) = \{0\}\times \r^{n-1}$.

\begin{theorem} \label{thm:partition}
Let $\{\Theta_1,\ldots,\Theta_M\}\in\mathcal{P}_M^\Gamma$ be a solution to the optimal $M$-partition problem \eqref{eq:4a}. Then, the following statements hold true.
\begin{itemize}
\item[$(i)$] There exist $a_1,\ldots,a_{M-1}\in(0,\pi)$ such that
$$(0,\pi)\smallsetminus\bigcup_{i=1}^M\widetilde{q}\,(\Theta_i)=\{a_1,\ldots,a_{M-1}\}.$$
Therefore, after reordering,
\begin{align*}
\Theta_1\cup(\mathbb{S}^{m-1}\times\{0\})&=\widetilde{q}\,^{-1}[0,a_1),\\
\Theta_i&=\widetilde{q}\,^{-1}(a_{i-1},a_i)\qquad\text{if }\; i=2,\ldots,M-1,\\
\Theta_M\cup(\{0\}\times\mathbb{R}^{n-1})&=\widetilde{q}\,^{-1}(a_{M-1},\pi].
\end{align*}
\item[$(ii)$] Set $\Omega_1 :=\Theta_1\cup(\mathbb{S}^{m-1}\times\{0\})$, $\Omega_M:=\Theta_M\cup(\{0\}\times\mathbb{R}^{n-1})$, and $\Omega_i:=\Theta_i$ otherwise. Then, $\Omega_1,\ldots,\Omega_M$ are smooth and connected, they satisfy $(c_1)$ and $(c_2)$ of \emph{Theorem} \ref{thm:main_RN}, $\Omega_1,\ldots,\Omega_{M-1}$ are bounded, $\Omega_M$ is unbounded, $\overline{\Omega_1\cup\cdots\cup \Omega_M}=\mathbb{R}^{N}$, and $\{\Omega_1,\ldots,\Omega_M\}\in\mathcal{P}_M^\Gamma$ is a solution to the optimal $M$-partition problem \eqref{eq:4a}.
\item[$(iii)$] If $w_i\in\mathcal{M}_{\Omega_i}^\Gamma$ satisfies $w_i\geq 0$ and $J_{\Omega_i}(w_i)=c_{\Omega_i}^\Gamma:=\inf_{\mathcal{M}^\Gamma_{\Omega_i}}J_{\Omega_i}$, then
$$w:=\sum_{i=1}^M(-1)^{i-1}w_i$$
is a $\Gamma$-invariant sign-changing solution to the problem \eqref{eq:1a} with precisely $M$ nodal domains and $w$ has minimal energy among all such solutions.
\end{itemize}
\end{theorem}

\begin{proof}
$(i):$ Note that Lemma \ref{lem:compactness} implies, by a standard argument, that $c_{\Omega}^\Gamma:=\inf_{\mathcal{M}^\Gamma_{\Omega}}J_{\Omega}$ is attained for any $\Gamma$-invariant smooth open subset $\Omega$ of $\mathbb{R}^N$ and we may assume the minimizer is strictly positive in $\Omega$.

Let $a,b,c\in(0,\pi)$ with $a<b<c$ and set $\Lambda_1:=\widetilde{q}\,^{-1}(a,b)$,\; $\Lambda_2:=\widetilde{q}\,^{-1}(b,c)$,\; $\Lambda=\widetilde{q}\,^{-1}(a,c)$. Then,
$$c_{\Lambda}^\Gamma < \min\{c_{\Lambda_1}^\Gamma,c_{\Lambda_2}^\Gamma\}.$$
Therefore, if $\{\Theta_1,\ldots,\Theta_M\}\in\mathcal{P}_M^\Gamma$ is a solution to the optimal $M$-partition problem \eqref{eq:4a}, then $(0,\pi)\smallsetminus\bigcup_{i=1}^M\widetilde{q}\,(\Theta)$ must consist of precisely $M-1$ points.

$(ii):$ Clearly, $\Omega_1,\ldots,\Omega_M$ are smooth and connected, they satisfy $(c_1)$ and $(c_2)$ of Theorem \ref{thm:main_RN}, $\Omega_1,\ldots,\Omega_{M-1}$ are bounded, $\Omega_M$ is unbounded, $\mathbb{R}^{N}=\overline{\Omega_1\cup\cdots\cup \Omega_M}$, and $\{\Omega_1,\ldots,\Omega_M\}\in\mathcal{P}_M^\Gamma$. As $m+(n-1)=N$ and $m,n\geq 2$, the codimension of \;$\mathbb{S}^{m-1}\times\{0\}$\; and \;$\{0\}\times\mathbb{R}^{n-1}$\; in $\mathbb{R}^N$ is at least $2$, so each one of these sets has capacity $0$ in $\mathbb{R}^N$; see \cite[Section 4.7]{eg}. Hence, $D_0^{1,2}(\Omega_i)=D_0^{1,2}(\Theta_i)$ and $c_{\Omega_i}^\Gamma=c_{\Theta_i}^\Gamma$.

$(iii):$ For each $i=1,\ldots,M-1$, let $\Phi_i$ be the interior of the set $\overline{\Omega}_i\cup \overline{\Omega}_{i+1}$. This is a $\Gamma$-invariant smooth domain in $\mathbb{R}^N$. Let $J_{\Phi_i}$ and $\mathcal{M}_{\Phi_i}^\Gamma$ be the energy functional and the Nehari manifold associated to the problem
\begin{equation} \label{eq:3b}
-\Delta w = |w|^{2^*-2}w, \qquad w\in D^{1,2}_0(\Phi_i)^\Gamma,
\end{equation}
see Section \ref{sec:euclidean}. The sign-changing solutions to \eqref{eq:3b} belong to the set
$$\mathcal{E}_{\Phi_i}^\Gamma:=\{w\in D^{1,2}_0(\Phi_i)^\Gamma:w^+\in\mathcal{M}_{\Phi_i}^\Gamma,\;w^-\in\mathcal{M}_{\Phi_i}^\Gamma\},$$
where $w^+:=\max\{w,0\}$ and $w^-:=\min\{w,0\}$. Lemma \ref{lem:compactness} implies that $J_{\Phi_i}$ satisfies the Palais-Smale condition on $\mathcal{M}_{\Phi_i}^\Gamma$. So, arguing as in \cite{ccn}, we see that every minimizer of $J_{\Phi_i}$ on $\mathcal{E}_{\Phi_i}^\Gamma$ is a solution to \eqref{eq:3b} and that $d_{\Phi_i}^\Gamma:=\inf_{\mathcal{E}_{\Phi_i}^\Gamma}J_{\Phi_i}$ is attained at some function $\widehat{w}_i\in\mathcal{E}_{\Phi_i}^\Gamma$. Setting 
$$\Phi_i^+:=\{x\in\Phi_i:\widehat{w}_i>0\}\qquad\text{and}\qquad\Phi_i^-:=\{x\in\Phi_i:\widehat{w}_i<0\},$$ we have that $\{\Omega_j:j\neq i,i+1\}\cup\{\Phi_i^+,\Phi_i^-\}\in\mathcal{P}_M^\Gamma$. 

Let $w_i\in\mathcal{M}_{\Omega_i}^\Gamma$ satisfy $w_i\geq 0$ and $J_{\Omega_i}(w_i)=c_{\Omega_i}^\Gamma$. Then, as $\Omega_i\cap\Omega_{i+1}=\emptyset$, we have that
$$\widetilde{w}_i:=(-1)^{i-1}w_i+(-1)^iw_{i+1}\in\mathcal{E}_{\Phi_i}^\Gamma.$$
We claim that $J_{\Phi_i}(\widetilde{w}_i)=d_{\Phi_i}^\Gamma$. Otherwise, since $c_{\Phi_i^\pm}^\Gamma \le J_{\Phi_i^\pm}(\widehat w^\pm)$,
$$c_{\Omega_i}^\Gamma+c_{\Omega_{i+1}}^\Gamma=J_{\Phi_i}(\widetilde{w}_i)>d_{\Phi_i}^\Gamma\geq c_{\Phi_i^+}^\Gamma + c_{\Phi_i^-}^\Gamma,$$
contradicting the fact that $\{\Omega_1,\ldots,\Omega_M\}$ solves the optimal $M$-partition problem \eqref{eq:4a}. Consequently, $J_{\Phi_i}(\widetilde{w}_i)=d_{\Phi_i}^\Gamma$. Since $\widetilde{w}_i$ solves \eqref{eq:3b}, we have that $\widetilde{w}_i\in\mathcal{C}^2(\Phi_i)$. Hence, $\widetilde{w}_i$ is a classical solution to \eqref{eq:3b} for every $i=1,\ldots,M-1$. Therefore, 
$$w=\sum_{i=1}^M(-1)^{i-1}w_i$$
is a classical solution to problem \eqref{eq:1a}. 

Finally, if $v\in D^{1,2}(\mathbb{R}^N)^\Gamma$ is a least energy solution to \eqref{eq:1a} with $M$ nodal domains $\Omega'_1,\ldots,\Omega'_M$, then, as $\{\Omega_1,\ldots,\Omega_M\}$ solves the optimal $M$-partition problem \eqref{eq:4a}, we have that
$$J(v)\geq \sum_{i=1}^Mc_{\Omega'_i}^\Gamma\geq \sum_{i=1}^Mc_{\Omega_i}^\Gamma=J(w).$$
Hence, $w$ has minimal energy.
\end{proof}

\begin{theorem} \label{thm:phase_separation}
For each $i,j=1,\ldots,M$, $i\neq j$, let $(\lambda_{ij,k})$ be a sequence of negative numbers such that $\lambda_{ij,k}\to -\infty$ as $k\to\infty$, and  let $v_{k}=(v_{k,1},\ldots,v_{k,M})$ be a positive least energy fully nontrivial $\Gamma$-invariant solution to the system \eqref{eq:2a} with $\lambda_{ij}=\lambda_{ij,k}$. Then, after passing to a subsequence, we have that
\begin{itemize}
\item[$(a)$]$v_{k,i}\to v_{\infty,i}$ strongly in $D^{1,2}(\mathbb{R}^N)$,\, $v_{\infty,i}\geq 0$,\, $v_{\infty,i}$ is continuous in $\mathbb{R}^N$ and $v_{\infty,i}|_{\Omega_i}$ is a least energy solution to the problem \eqref{eq:3a} in $\Omega_i:=\{x\in\mathbb{R}^{N}:v_{\infty,i}(x)>0\}$, for each $i=1,\ldots,M$.
\item[$(b)$]$\{\Omega_1,\ldots,\Omega_M\}\in\mathcal{P}^\Gamma_M$ and it solves the optimal $M$-partition problem \eqref{eq:4a}.
\end{itemize}
\end{theorem}

\begin{proof}
To highlight the role of $\lambda_{ij,k}$, we write $\mathcal{J}_k$ and $\mathcal{N}^\Gamma_k$ for the functional $\mathcal{J}$ and the set $\mathcal{N}^\Gamma$ associated to the system \eqref{eq:2a} with $\lambda_{ij}=\lambda_{ij,k}$; see Section \ref{sec:euclidean}. By assumption,
$$c_k^\Gamma:= \inf_{\mathcal{N}^\Gamma_k} \mathcal{J}_k =\mathcal{J}_k(v_k)=\frac{1}{N}\sum_{i=1}^M\|v_{k,i}\|^2.$$
We define
\begin{align*}
\mathcal{N}_0^\Gamma:=\{(v_1,\ldots,v_M)\in\mathcal{H}:\,&v_i\neq 0,\;\|v_i\|^2=|v_i|_{2^*}^{2^*},\\
& \text{ and }v_iv_j=0\text{ a.e. in }\mathbb{R}^N \text{ if }i\neq j\}.
\end{align*}
Then, $\mathcal{N}_0^\Gamma\subset\mathcal{N}^\Gamma_k$ for all $k\in\mathbb{N}$ and, consequently, 
$$0<c_k^\Gamma\leq c_0^\Gamma:=\inf\left\{\frac{1}{N}\sum_{i=1}^M\|v_i\|^2:(v_1,\ldots,v_M)\in\mathcal{N}_0^\Gamma\right\}<\infty.$$
So, after passing to a subsequence, using Lemma \ref{lem:compactness} we get that $v_{k,i} \rightharpoonup v_{\infty,i}$ weakly in $D_{0}^{1,2}(\mathbb{R}^N)^\Gamma$, $v_{k,i} \to v_{\infty,i}$ strongly in $L^{2^*}(\mathbb{R}^N)$ and $v_{k,i} \to v_{\infty,i}$ a.e. in $\mathbb{R}^N$, for each $i=1,\ldots,M$. Hence, $v_{\infty,i} \geq 0$. Moreover, as $\partial_i\mathcal{J}_k(v_k)[v_{k,i}]=0$, we have that, for each $j\neq i$,
\begin{align*}
0&\leq\int_{\mathbb{R}^N}\beta_{ij}|v_{k,j}|^{\alpha_{ij}}|v_{k,i}|^{\beta_{ij}}\leq \frac{|v_{k,i}|^{2^*}_{2^*}}{-\lambda_{ij,k}}\leq \frac{C}{-\lambda_{ij,k}}.
\end{align*}
Then, Fatou's lemma yields 
$$0 \leq \int_{\mathbb{R}^N}|v_{\infty,j}|^{\alpha_{ij}}|v_{\infty,i}|^{\beta_{ij}} \leq \liminf_{k \to \infty} \int_{\mathbb{R}^N}|v_{k,j}|^{\alpha_{ij}}|v_{k,i}|^{\beta_{ij}} = 0.$$
Hence, $v_{\infty,j} v_{\infty,i} = 0$ a.e. in $\mathbb{R}^N$. On the other hand, as shown in \cite[Proposition~3.1]{cs}, using Sobolev's inequality we see that
$$0<d_0 \leq \|v_{k,i}\|^2 \leq |v_{k,i}|_{2^*}^{2^*}\qquad\text{for all }k\in\mathbb{N},\;i=1,\ldots,M.$$
So, as $v_{k,i} \to v_{\infty,i}$ strongly in $L^{2^*}(\mathbb{R}^N)$, we conclude that $v_{\infty,i}\neq 0$. And, as $v_{k,i} \rightharpoonup v_{\infty,i}$ weakly in $D^{1,2}(\mathbb{R}^N)$, we get that
\begin{equation} \label{eq:comparison2}
\|v_{\infty,i}\|^2 \leq |v_{\infty,i}|_{2^*}^{2^*}\qquad\text{for all }i=1,\ldots,M.
\end{equation}
Since $v_{\infty,i}\neq 0$, there is a unique $t_i\in(0,\infty)$ such that $\|t_iv_{\infty,i}\|^2 = |t_iv_{\infty,i}|_{2^*}^{2^*}$. Then, $(t_1v_{\infty,1},\ldots,t_Mv_{\infty,M})\in \mathcal{N}_0^\Gamma$. The inequality \eqref{eq:comparison2} implies that $t_i\in (0,1]$. Therefore,
\begin{align*}
c_0^\Gamma &\leq \frac{1}{N}\sum_{i=1}^M\|t_iv_{\infty,i}\|^2 \leq \frac{1}{N}\sum_{i=1}^M\|v_{\infty,i}\|^2\\
&\leq \frac{1}{N}\liminf_{k\to\infty}\sum_{i=1}^M\|v_{k,i}\|^2=\liminf_{k\to\infty} c_k^\Gamma \leq c_0^\Gamma.
\end{align*}
Hence, $v_{k,i} \to v_{\infty,i}$ strongly in $D^{1,2}(\mathbb{R}^N)^\Gamma$,\; $t_i=1$, yielding 
\begin{equation}\label{eq:limit}
\|v_{\infty,i}\|^2 = |v_{\infty,i}|_{2^*}^{2^*},\qquad\text{and}\qquad\frac{1}{N}\sum_{i=1}^M\|v_{\infty,i}\|^2 = \lim_{k\to\infty} c_k^\Gamma.
\end{equation}

Set $Y_1:=\mathbb{S}^{m-1}\times\{0\}$, $Y_2:=\{0\}\times\mathbb{R}^{n-1}$, and $Y:=Y_1\cup Y_2$. Proposition \ref{prop:continuity}, together with Lemma \ref{lem:isometry1}, imply that $v_{\infty,i}|_{\mathbb{R}^N\smallsetminus Y}$ is continuous. Consequently, $\Theta_i:=\{x\in\mathbb{R}^{N}\smallsetminus Y:v_{\infty,i}(x)>0\}$ is $\Gamma$-invariant and open in $\mathbb{R}^N$. Since $v_{\infty,i}\neq 0$ and $v_{\infty,i}v_{\infty,j}=0$ if $i\neq j$, we have that $\{\Theta_1,\ldots,\Theta_M\}\in\mathcal{P}_M^\Gamma$. As we have already noticed (see the proof of (ii) of Theorem \ref{thm:partition}), $Y$ has capacity 0 in $\rn$. Hence, from \eqref{eq:limit} and \eqref{eq:comparison} we get that $v_{\infty,i}|_{\Theta_i}\in\mathcal{M}_{\Theta_i}^\Gamma$ and
\begin{equation*}
\sum_{i=1}^M c_{\Theta_i}^\Gamma\leq\frac{1}{N}\sum_{i=1}^M\|v_{\infty,i}\|^2 = \lim_{k\to\infty} c_k^\Gamma \leq \inf_{(\Phi_1,\ldots,\Phi_M)\in\mathcal{P}_M^\Gamma}\;\sum_{i=1}^M c_{\Phi_i}^\Gamma.
\end{equation*}
This shows that $\{\Theta_1,\ldots,\Theta_M\}$ solves the optimal $M$-partition problem \eqref{eq:4a}. 
 
Reordering this partition as indicated in Theorem \ref{thm:partition}, and setting $\Omega_1:=\Theta_1\cup Y_1$, $\Omega_M:=\Theta_M\cup Y_2$ and $\Omega_i:=\Theta_i$ if $i\neq 1,M$, we have that $\{\Omega_1,\ldots,\Omega_M\}\in\mathcal{P}_M^\Gamma$ and $c_{\Omega_i}^\Gamma=c_{\Theta_i}^\Gamma$. As $v_{\infty,i}|_{\Omega_i}\in\mathcal{M}_{\Omega_i}^\Gamma$ and $J_{\Omega_i}(v_{\infty,i}|_{\Omega_i}) = c_{\Omega_i}^\Gamma$, the function $v_{\infty,i}|_{\Omega_i}$ solves problem \eqref{eq:3a} in $\Omega_i$. Since $\Omega_i$ is smooth by Theorem \ref{thm:partition}, we have that $v_{\infty,i}$ is continuous in $\mathbb{R}^N$ and $\Omega_i=\{x\in\mathbb{R}^{N}:v_{\infty,i}(x)>0\}$. This concludes the proof. 
\end{proof} \medskip

\begin{proof}[Proof of Theorem \ref{thm:main_RN}]
This result follows immediately from Theorems \ref{thm:partition} and \ref{thm:phase_separation}.
\end{proof}

The following result rephrases Corollary \ref{cor:main}.

\begin{corollary} \label{cor:main_RN}
\begin{itemize}
\item[$(i)$]There exists a solution to the optimal $M$-partition problem \eqref{eq:4a} in $\mathbb{R}^N$.
\item[$(ii)$]There exists a least energy $\Gamma$-invariant sign-changing solution to the problem \eqref{eq:1a} with precisely $M$ nodal domains.
\end{itemize}
\end{corollary}

\begin{proof}
The existence of a positive least energy fully nontrivial solution to the system \eqref{eq:2a} was established in \cite{cs}. So these statements follow from Theorem~\ref{thm:main_RN}. 
\end{proof}

We conclude with the following result which, together with Theorem \ref{thm:partition}, establishes a close relationship between solutions to the optimal $M$-partition problem \eqref{eq:4a} and least energy $\Gamma$-invariant sign-changing solutions to the problem \eqref{eq:1a} with precisely $M$ nodal domains.

\begin{corollary} \label{cor:nodal}
If $v\in D^{1,2}(\mathbb{R}^N)$ is a $\Gamma$-invariant sign-changing solution to the problem \eqref{eq:1a} with precisely $M$ nodal domains and $v$ has minimal energy among all such solutions, then its nodal domains $\{\Omega_1,\ldots,\Omega_M\}$ satisfy the optimal $M$-partition problem \eqref{eq:4a} in $\mathbb{R}^N$.
\end{corollary}

\begin{proof}
By Corollary \ref{cor:main_RN}, there exists a solution $\{\Theta_1,\ldots,\Theta_M\}$ to the optimal $M$-partition problem \eqref{eq:4a} and, by Theorem \ref{thm:partition} there exists a $\Gamma$-invariant sign-changing solution $w$ to \eqref{eq:1a}, with precisely $M$ nodal domains, such that
$$J(w)= \sum_{i=1}^M c_{\Theta_i}^\Gamma.$$
Now, we argue by contradiction. Let $v$ be a least energy $\Gamma$-invariant sign-changing solution to \eqref{eq:1a} with precisely $M$ nodal domains. If the set of its nodal domains $\{\Omega_1,\ldots,\Omega_M\}$ were not a solution to the optimal $M$-partition problem \eqref{eq:4a}, we would have that 
$$J(w)= \sum_{i=1}^M c_{\Theta_i}^\Gamma=\inf_{\{\Lambda_1,\ldots,\Lambda_M\}\in\mathcal{P}_M^\Gamma}\;\sum_{i=1}^M c_{\Lambda_i}^\Gamma<\sum_{i=1}^M c_{\Omega_i}^\Gamma\leq J(v).$$
This is a contradiction. 
\end{proof}

\begin{remark}
The argument used to prove $(i)$ in \emph{Theorem }\ref{thm:partition} shows that the expression \eqref{eq:4a} is increasing in $M$. So \emph{Corollary} \ref{cor:nodal} implies that, if a $\Gamma$-invariant $M$-nodal solution to \eqref{eq:1a} has minimal energy among all $\Gamma$-invariant $M$-nodal solutions, it has also minimal energy among all $\Gamma$-invariant solutions with at least $M$ nodal domains.
\end{remark}

\bigskip

\begin{flushleft}
\textbf{Mónica Clapp}\\
Instituto de Matemáticas\\
Universidad Nacional Autónoma de México\\
Circuito Exterior, Ciudad Universitaria\\
04510 Coyoacán, Ciudad de México, Mexico\\
\texttt{monica.clapp@im.unam.mx} 
\medskip

\textbf{Alberto Saldaña}\\
Instituto de Matemáticas\\
Universidad Nacional Autónoma de México\\
Circuito Exterior, Ciudad Universitaria\\
04510 Coyoacán, Ciudad de México, Mexico\\
\texttt{alberto.saldana@im.unam.mx} 
\medskip

\textbf{Andrzej Szulkin}\\
Department of Mathematics\\
Stockholm University\\
106 91 Stockholm, Sweden\\
\texttt{andrzejs@math.su.se} 
\end{flushleft}

\end{document}